\documentclass[10pt]{amsart}
\usepackage{latexsym, amssymb, amscd, amsfonts,pstricks,enumitem,amsmath, young}
\usepackage{latexsym, amssymb, amscd, amsfonts,pstricks,enumitem,amsmath}
\usepackage{amsmath,amsthm,amssymb,mathrsfs}
\usepackage{amsmath, amsfonts, amssymb,amsthm}
\usepackage{amstext}
\usepackage{mathrsfs}
\usepackage{tikz-cd}
\usepackage{fancyhdr}
\usepackage{verbatim}

\numberwithin{equation}{section}

\theoremstyle{plain}
\newtheorem{theorem}{Theorem}[section]

\newtheorem{def-thm}[theorem]{Definition-Theorem}
\newtheorem{lemma}[theorem]{Lemma}

\newtheorem{definition}[theorem]{Definition}

\theoremstyle{definition}
\newtheorem{remark}[theorem]{Remark}
\newtheorem{example}[theorem]{Example}

\newtheorem{conjecture}[theorem]{Conjecture}

\newcommand{\sq}[1]{\ifx#1([\else\ifx#1)]%
  \else\message{invalid use of "sq"}\fi\fi}

\DeclareMathOperator{\Supp}{Supp}

\DeclareMathSymbol{\idot}{\mathbin}{operators}{`\.}
\allowdisplaybreaks
\makeatother

\begin{document}
\title{ Hyperbolicity and GCD for $n+1$  divisors with non-empty intersection}
 
\author{Julie Tzu-Yueh Wang}
\address{Institute of Mathematics, Academia Sinica \newline
\indent No.\ 1, Sec.\ 4, Roosevelt Road\newline
\indent Taipei 10617, Taiwan}
\email{tywang@as.edu.tw}

\author{Zheng Xiao}
\address{Department of Mathematics, University of Colorado, Boulder
  \newline
\indent  2300 Colorado Avenue \newline
\indent  Boulder, CO 80309 USA} 
\email{zheng.xiao@colorado.edu}

\thanks{2020\ {\it Mathematics Subject Classification}: Primary 32H30; Secondary 32Q45 and 30D35}
\thanks{The  first-named author was supported in part by Taiwan's NSTC grant  113-2115-M-001-011-MY3.}

\begin{abstract} 
We study hyperbolicity  for quasi-projective varieties where the boundary divisor consists of $n+1$ numerically parallel effective divisors on a complex projective variety of dimension 
$n$, allowing non-empty intersection. Under explicit local conditions on 
$\beta$-constants or intersection multiplicities, we prove that all entire curves are algebraically degenerate.

Our approach extends the method of Levin–Huang–Xiao \cite{HLX} to higher dimensions, establishing a second main theorem for regular sequences of closed subschemes.  This also yields a GCD-type estimate in the same geometric setting.
\end{abstract}

\maketitle
\baselineskip=16truept

 \section{Introduction}\label{sec:intro}

The celebrated Green-Griffiths-Lang conjecture predicts the hyperbolic behavior of entire curves on quasi-projective varieties of (log-)general type:
\begin{conjecture}[Green-Griffiths-Lang Conjecture]\label{GGL}
Let $X$ be a smooth complex projective variety, and let $D$ be a normal crossings divisor on $X$.   Suppose that $X\setminus D$ is a variety of log-general type.   Then:   
\begin{enumerate}
\item
Every entire curve $f : \mathbb{C} \to X\setminus D$ is algebraically degenerate; that is, the image of $f$ is contained in a proper algebraic subvariety of $X$.
\item Furthermore, there exists a proper algebraic subvariety $Z \subset X$ such that the image of every entire curve $f : \mathbb{C} \to X$ is contained in $Z$
\end{enumerate}
\end{conjecture}
Part (ii) is commonly referred to as the \emph{strong Green-Griffiths-Lang conjecture} (see \cite[Conjecture 6.6.30]{NW}).
 
Substantial progress has been made toward this conjecture in various settings. In this paper, we focus on the case where $D$ is the union of  $n+1$ numerically parallel effective divisors. A collection of  effective divisors  $D_1,\hdots,D_q$ on  $X$, $q\ge 2$,  
is said to be {\it numerically parallel} if there exist positive integers $a_i$ such that $a_1D_1, \ldots, a_{n+1}D_{n+1}$ are all numerically equivalent.  A basic example is when $X=\mathbb P^n$, in which any collection of hypersurfaces of the same degree is numerically equivalent. 

 
 Conjecture \ref{GGL} (i) has been  proved by Noguchi, Winkelmann and Yamanoi  in \cite{NWY07} for the case where $X=\mathbb P^n$ and $D$ consists of $n+1$ normal crossings hypersurfaces.  Furthermore, the GCD method developed in  \cite{GNSW} provides an explicit formulation of the exceptional set $Z$ in Conjecture \ref{GGL} (ii).  Recently, Ru and the first author established 
in \cite{RW24}  a version of Campana orbifold conjecture in the setting of  $n+1$ normal crossings numerically parallel divisors on a smooth complex project variety, which in turn implies Conjecture \ref{GGL} (i) for that case.

 In this paper, we turn to a different setting, where $D$  consists of $n+1$ numerically parallel effective divisors that  have non-empty intersection.  
To state our main result in this direction, we first recall the definition of the $\beta$-constant for a closed subscheme (\cite[Definition 1.2]{RW22}), which is a direct generalization of Autissier's $\alpha$-constant (\cite[Definition 2.4]{Aut11}).
\begin{definition}\label{def_beta}
Let $\mathscr L$ be a big line sheaf and let $Y$ be a closed subscheme
 on a projective variety $X$.  We define
\begin{equation*}
  \beta(\mathscr L, Y)
    = \liminf_{N\to\infty}
      \frac {\sum_{m\ge 1}h^0(\mathscr L^N\otimes {\mathscr I}_{Y}^{m})} {Nh^0(\mathscr L^N)},
\end{equation*}
where ${\mathscr I}_{Y}$ is the ideal sheaf defining $Y$. 
\end{definition}

We adopt the following convention from \cite{HLX} for decomposing zero-dimensional subschemes:
\begin{definition}\label{local_beta}
Let  $Y$ be a zero-dimensional closed subscheme
 of a complex projective variety $X$ with ${\rm Supp} Y=\{Q_1,\hdots,Q_m\}$.  Then we write 
 $$
 Y=Y_{Q_1}+\cdots+Y_{Q_m},
 $$  
 where $Y_{Q_i}$ is a closed subscheme supported only at the point $Q_i$
 \end{definition}

Our main theorem relies on a notion of ``proper intersection" for divisors.
\begin{definition}
    Let $D_1,\ldots,D_q$ be effective Cartier divisors on a variety $X$. We say $D_1,\ldots,D_q$ intersect properly if for any $I \subset \{1,\ldots,q\}$ and any $x \in \bigcap_{i \in I}  D_i$, the sequence $(\phi_i)_{i \in I}$ is a regular sequence in the local ring $\mathcal{O}_{X,x}$, where $\phi_i$ are the local defining functions of $D_i$, $1 \leq i \leq q$.
\end{definition}
Of course, this notion of proper intersection can be generalized to closed subschemes; see, for example, \cite{Voj23}.

We now state our main theorem on hyperbolicity in the setting where the $n+1$ divisors are numerically parallel and have nontrivial intersection.

\begin{theorem}\label{mainTheorem2}
Let $X$ be a complex projective variety  of dimension $n$. 
Let $D_1, \ldots,D_{n+1}$ be effective Cartier  divisors on $X$ such that  there exist positive integers $a_1,\ldots, a_{n+1}$ such that $a_1D_1, \ldots, a_{n+1}D_{n+1}$ are all numerically equivalent to an ample divisor $D$.  
Assume that   any $n$-tuple of divisors among  $D_1, \ldots,D_{n+1}$ intersect properly, and that
\begin{align*}
\bigcap_{i=1}^{n+1}D_i\neq \emptyset.
\end{align*}
Suppose further that for every point 
 $Q\in  \bigcap_{i=1}^{n+1}D_i  $ and every subset $I \subset \{1,\ldots,n+1\}$, with $|I|=n$,  the following inequality holds:
 \begin{align}
\label{betacond}
\beta(D,(\bigcap_{i \in I}a_iD_i)_Q)>1
\end{align}
 Then there exists a proper Zariski-closed subset $Z\subset X$ such that the image of  any non-constant holomorphic map $f:\mathbb{C}\to X\setminus{\bigcup_{i=1}^{n+1} } D_i$  is contained in $Z$.
    \end{theorem}
\begin{remark}\label{localcondition}
\begin{enumerate}
\item The condition \eqref{betacond} on the $\beta$-constant is crucial and will be illustrated by examples in Section \ref{hyperbolic}. 
\item The inequality \eqref{betacond}  can be replaced by the following condition involving local multiplicities:
 \begin{align}\label{localmultip} 
   (D_i)^I_Q<\bigg(\frac{n}{n+1}\bigg)^n (D_i)^I, 
\end{align}
where $(D_i)^I_Q$ denotes the local intersection multiplicity of $D_i$ (for  $i\in I$) at the point  $Q$. We will establish this implication in Section~\ref{hyperbolic}.
 \end{enumerate}
\end{remark}

Allowing non-empty intersection introduces new geometric and analytic challenges, particularly in controlling the local behavior near the intersection locus. A more compelling perspective arises from studying the distribution of a set of integral points, which, in Vojta's dictionary relating Nevanlinna theory and Diophantine approximation, corresponds to the image of an entire curve.   As previously mentioned, various results have been established for Conjecture~\ref{GGL}, including the case where  $X=\mathbb P^n$ and $D$ consists of $n+1$ normal crossings hypersurfaces.  However, the analogous conjecture over number fields, namely, the Lang-Vojta conjecture, remains widely open. Notably, there are no known results even for $n=2$.     Nevertheless, Corvaja and Zannier \cite{CZ00}, as well as Levin, Huang, and the second author  \cite{HLX}, have obtained results on the degeneracy of integral points in the case of surfaces where 
$D$ consists of three numerically parallel divisors with non-empty intersection. The number field analogue of Theorem~\ref{mainTheorem2} holds and, in fact, extends the results of \cite{HLX}.  The proof of Theorem~\ref{mainTheorem2} builds on the  techniques  developed in \cite{HLX} for the surface case, together with a generalization of their main theorem to the setting of a regular sequence of closed subschemes of $X$, defined as follows:

\begin{definition}[{\cite[Definition 2.14]{HLX}}]
Let $Y_1\supset Y_2\supset \cdots\supset Y_m$ be closed subschemes of a projective variety $X$. We say that this is a {\it regular sequence of closed subschemes of $X$} if, for each $i\in \{1,\ldots, m\}$ and every point $P\in \Supp Y_i$, there exists a regular sequence $f_1,\ldots, f_i\in \mathcal{O}_{X,P}$ such that for $1\leq j\leq i$, the ideal sheaf of $Y_j$ is locally given by the ideal $(f_1,\ldots, f_j)$ in $\mathcal{O}_{X,P}$.
\end{definition}
\begin{remark}
    If divisors $D_1,\ldots, D_m$ intersect properly on $X$, then the chain of intersections $D_1\supset D_1\cap D_2\supset\cdots\supset \bigcap_{i=1}^m D_i$ forms a regular sequence of closed subschemes of $X$. Moreover, if $X$   is smooth and  $D_1,\ldots,D_m$ are in general position, then this chain of nested intersections is a regular sequence.
    \end{remark}

 The main theorem of \cite{HLX} refines the general theorem of Ru and Vojta \cite{RV20} in the surface case (see Theorem~\ref{generalRV}), extending the setting from divisors $D_i$ to sequences of closed subschemes such as $D_i \supset D_i \cap D_j$. This broader framework allowed the authors to establish several significant Diophantine results in \cite{HLX} and suggests a promising direction for further investigation.

We now state our generalization of this result, namely \cite[Theorem~1.4]{HLX}, to higher-dimensional settings.

 \begin{theorem}\label{trungeneral}
    Let $X$ be a complex projective variety of dimension $n$, and let $D_1,\cdots,D_q$  be effective Cartier divisors on  $X$.
  Let $\mathcal M$ be a finite collection of pairs $(I,Y)$, where 
$I \subset \{1,\ldots,q\}$ has cardinality $n-1$, $Y$ is a closed subscheme of $\bigcap_{i \in I} D_i$, and the divisors $D_i$, $i \in I$, intersect properly  such that $ \bigcap_{i \in I} D_i\supset Y $
is a  regular sequence. Let $\mathscr L$ be a big line sheaf on $X$.  
 Then, for every $\varepsilon>0$,  there exists a proper  Zariski-closed subset $Z\subset X$ such that for any non-constant holomorphic map $f:\mathbb{C}\to X$ with  $f(\mathbb{C})\not\subset Z$,  the following  inequality  holds:
 \begin{align}\label{mainineq}
\int_0^{2\pi} \max_{(I,Y) \in \mathcal{M}} \Bigg(
&\sum_{i \in I} \beta(\mathscr{L}, D_i)\lambda_{D_i}(f(re^{i\theta}))  
 + \Big( \beta(\mathscr{L}, Y) - \sum_{i \in I} \beta(\mathscr{L}, D_i) \Big) \lambda_Y(f(re^{i\theta}))
\Bigg)   \frac{d\theta}{2\pi}\notag\\
&\leq_{\operatorname{exc}} (1+\varepsilon) T_{\mathscr{L}, f}(r),
\end{align}
where the maximum is taken over all   $(I,Y)\in\mathcal M$, and the notation $\leq_{\operatorname{exc}}$ means the inequality holds for all $r\in \mathbb{R}^+$ except a set of finite Lebesgue measure.
\end{theorem}

Finally, as in the surface case treated in \cite{HLX}, the above theorem can be applied to derive a GCD-type result, which we state below:
\begin{theorem}\label{MainThmgcd}
Let $D_1, \ldots,D_{n+1}$ be effective  divisors intersecting properly on a complex projective variety $X$ of dimension $n$. Suppose that there exist positive integers $a_1,\ldots, a_{n+1}$ such that $a_1D_1, \ldots, a_{n+1}D_{n+1}$ are all numerically equivalent to an ample divisor $D$.  
Suppose that for some index set $I_0 \subset \{1,\ldots,n+1\}$ with $|I_0|=n$ such that $\bigcap_{i\in I_0}D_i$ contains more than one point. 
Then, for  each $\varepsilon>0$,   there exists a proper  Zariski-closed subset $Z\subset X$ such that for any non-constant holomorphic map $f:\mathbb{C}\to X\setminus{\bigcup_{i=1}^{n+1} } D_i$ such that  $f(\mathbb{C})\not\subset Z$,  we have 
\begin{align*}
T_{ \bigcap_{i \in I_0}a_iD_i, f}(r) 
        \le \varepsilon  T_{D,f}(r). 
\end{align*}
\end{theorem}
We refer the reader to \cite{HLX} for an overview of various GCD estimates, including results of Levin \cite{LevinGCD}, Levin and the first author \cite{LW}, the first author and Yasufuku \cite{WY}, and Huang and Levin \cite{HL}.

The proof of Theorem~\ref{trungeneral} builds on the work of \cite{HLX}, but with a different choice of filtrations.  To illustrate the core idea behind this approach, let us consider the case where \( X \) is a surface and \( P \) is a point lying in the intersection \( D_1 \cap D_2 \) of two effective divisors. Let \( \mathcal{L} \) be a big line sheaf on \( X \). In this setting, \cite{HLX} constructs filtrations on the spaces \( H^0(X, \mathcal{L}^N \otimes \mathcal{O}(D_i)) \), for \( i = 1, 2 \), in such a way that the resulting sections vanish to high order along \( D_i \). A separate filtration is applied to \( H^0(X, \mathcal{L}^N \otimes {\mathscr I}_{Y}) \), where \( {\mathscr I}_{Y} \) is the ideal sheaf of \( Y = D_1 \cap D_2 \), to control vanishing along the intersection.
In our generalization to higher dimensions, the filtrations on  \( H^0(X, \mathcal{L}^N \otimes \mathcal{O}(D_i)) \) are replaced with the moving filtration introduced by Autissier in \cite{Aut11}. 

The complete proof of Theorem~\ref{trungeneral}, including the necessary modifications to several technical lemmas from \cite{HLX}, will be presented in Section~\ref{ProofMain}. In Section~\ref{Nevanlinna}, we review basic notions and results from Nevanlinna theory for divisors and closed subschemes. Section~\ref{hyperbolic} is devoted to formulating and proving a refined version of Theorem~\ref{mainTheorem2}, along with examples illustrating the necessity of the conditions on the $\beta$ constant and local intersection multiplicities.  The proof of Theorem~\ref{MainThmgcd} will be given in Section~\ref{GCD}.  Finally, we remark that our results extend analogously to the number field setting; this will be addressed in a forthcoming sequel focusing on arithmetic applications.

 \section{Preliminary in Nevanlinna Theory}\label{Nevanlinna} 
  We first recall some definitions in Nevanlinna theory.
  Let $D$ be an effective Cartier divisor on a complex variety $X$.  Let  $s=1_D$ be a canonical section of $\mathscr O(D)$ (i.e. a global section for which $(s)=D$).  Choose a smooth metric $|\cdot|$
on $\mathscr O(D)$.  The associated Weil function $ \lambda_D: X(\mathbb C)\setminus\Supp D$ is given by
\begin{equation*} 
  \lambda_D(x) := -\log|s(x)|.
\end{equation*}
It is linear in $D$ (over a suitable domain), so by linearity and continuity
it can be extended to a definition of $\lambda_D$ for a general Cartier divisor
$D$ on $X$.  The definition of Weil functions extends to closed subschemes.  In particular, we may write $Y$ as an intersection of divisors, i.e. $Y=\cap_{i=1}^{\ell} D_i$, then $\lambda_Y=\min_{i=1}^{\ell}\{\lambda_{D_i}\} +O(1).$   We refer to \cite{Yamanoi} for the construction and properties. 

Let $f:\mathbb{C}\to X$ be a holomorphic map whose image is not contained in the support of closed subscheme  $Y$ on $X$. 
 The {\it proximity function} of $f$ with respect to $Y$ is defined by $$m_f(Y,r)=\int_0^{2\pi}\lambda_Y(f (re^{i\theta}))\frac{d\theta}{2\pi}.$$
 Let 
$$
{\rm ord}_zf^*Y:=\min\{{\rm ord}_zf^*D_1,\hdots,{\rm ord}_zf^*D_{\ell}\}.
$$
When $f(\mathbb{C})\subset D_i$, we set ${\rm ord}_zf^*D_{i}=\infty$.
The  {\it counting function}   is defined  by
$$
N_f(Y,r) = \int_1^r \left(\sum_{z\in B(t)}{\rm ord}_zf^*Y \right) \frac{dt}{t},
$$
where $B(t)=\{z\in \mathbb C\, : |z|<t\}.$   
 The characteristic function relative to $Y$ is defined, up to $O(1)$, as 
\begin{align}\label{FMT}
T_{Y,f}(r):=m_f(Y,r)+N_f(Y,r).
\end{align}

We recall the following standard result, in the spirit of the Second Main Theorem, from \cite[Theorem 7.3]{RTW21} (see also \cite[Theorem A.7.1.1]{Rubook}).
\begin{theorem}
\label{gsmt2}
    Let $X$ be a complex projective variety and let $D$ be a Cartier divisor on $X$, let $V$ be a nonzero $d$-dimensional linear subspace of $H^0(X,\mathcal{O}(D))$, and let $s_1,\dots,s_q$ be nonzero elements of $V$. 
 Let $\Phi=(\phi_1,\hdots,\phi_{d}):X\dashrightarrow \mathbb{P}^{d-1}$ be the rational map associated to the linear system $V$. For each $j=1,\dots,q$, let $D_j$ be the Cartier divisor $(s_j)$. 
     Let $\Psi=(\psi_1,\hdots,\psi_{d}):\mathbb C\to   \mathbb{P}^{d-1}$ be a reduced form of $\Phi\circ f$, i.e.  $\Psi=\Phi\circ f$  and $\psi_1,\hdots,\psi_{d}$ are entire functions without common zeros. 
    Then for any $\varepsilon>0$, there exists a proper  Zariski-closed subset such that for any nonconstant holomorphic map $f:\mathbb{C}\to X$ such that  $f(\mathbb{C})\not\subset Z$, the following holds:
       $$\int_0^{2\pi} \max_J \sum_{j\in J} \lambda_{D_j}(f(re^{i\theta}))\frac{d\theta}{2\pi} \leq_{\operatorname{exc}} (\operatorname{dim}V+\varepsilon)T_{D,f}(r),$$ where $J$ ranges over all subsets of $\{1,\dots,q\}$ such that the sections $(s_j)_{j\in J}$ are linearly independent.
       \end{theorem}

For the convenience of the reader, we state the following general theorem of Ru-Vojta in complex case.
\begin{theorem}\label{generalRV}
 Let $X$ be a complex projective variety of dimension $n$,  let $\mathscr L$ be a big line sheaf on $X$, and let
 $D_1,\cdots,D_q$  be effective Cartier divisors on  $X$ that intersect properly.
 Then, for every $\varepsilon>0$,  there exists a proper  Zariski-closed subset $Z\subset X$ such that for any non-constant holomorphic map $f:\mathbb{C}\to X$ with  $f(\mathbb{C})\not\subset Z$,  the following  inequality  holds:
\begin{align}\label{mainineqRV}
 \sum_{i=1}^q \beta(\mathscr L, D_i)m_f(D_i,r)
 \leq_{\operatorname{exc}} (1+\varepsilon)T_{\mathscr L,f}(r).
  \end{align}
\end{theorem}

\baselineskip=16truept \maketitle \pagestyle{myheadings}
\section{Proof of Theorem \ref{trungeneral}}\label{ProofMain}
\subsection{Some Technical Lemmas}
 
 We begin by recalling some technical lemma from \cite{HLX} with minor modifications. 
\begin{definition}
Let $I$ and $J$ be ideals of a ring $R$.  The quotient ideal $(I:J)$ is defined by $(I:J)=\{r\in R\,|\, rJ\subset I\}$.
\end{definition}
\begin{lemma}[cf. {\cite[Lemma 2.9]{HLX}}]\label{ideallemma}
    Let $I=(f_1,\ldots, f_m)$ be an ideal in a ring $R$, generated by a regular sequence $f_1,\ldots, f_m$. 
Let  $A=\{n_1,\ldots, n_m\} $ be an $m$-tuples of nonnegative integers, and let $|A|=n_1+\cdots+n_m$.  Define $\mathbf{f}^A:=f_1^{n_1}\cdots f_m^{n_m}$.  Then for any integer $n\ge |A|$, we have
\begin{align*}
(I^n:(\mathbf{f}^A))=I^{n-|A|}.
\end{align*}

\end{lemma}

\begin{proof}
Since $\mathbf{f}^A \in I^{|A|}$, the inclusion $I^{n-|A|} \subset (I^n:(\mathbf{f}^A))$ is clear. 

To prove the reverse inclusion, suppose that $a\mathbf{f}^A   \in I^n$ for some nonzero $a\in R$. Consider the associated graded ring $G_I(R) := \bigoplus_{r \ge 0} I^r / I^{r+1}$. By \cite[Theorem 27]{Mat80}, since $f_1,\ldots, f_m$ forms a regular sequence, the graded ring $G_I(R)$ is isomorphic to the polynomial ring $(R/I)[\mathbf{x}]=(R/I)[x_1,\ldots, x_m]$, where the indeterminates $x_1,\ldots, x_m$ correspond to the images of $f_1,\ldots, f_m$ in the first graded piece of $G_I(R)$.
Suppose $a \in I^\ell $  for some nonnegative integer $\ell$, so that the image $  \bar{a} \in I^\ell / I^{\ell+1}$ is nonzero. Since each $x_j$ is not a zero divisor in $(R/I)[x_1,\ldots, x_m]$,  the monomial $\mathbf{x}^A:=x_1^{n_1}\cdots x_m^{n_m} $ is also not a zero divisor. It follows that
$0 \neq \overline{a \mathbf{f}^A} \in I^{\ell + | A|} / I^{\ell + | A| + 1}$. Therefore $a \mathbf{f}^A \notin  I^{\ell + |A| + 1}$.  Since we are assuming $a\mathbf{f}^A   \in I^n$, it must be  $\ell + |A| + 1 \ge n+1$, i.e. $\ell \ge n - | A|$.   Thus, $a\in I^\ell \subset I^{n - | A|},$ and we conclude that $(I^n:(\mathbf{f}^A))\subset I^{n-|A|}.$
\end{proof} 
\begin{definition}
Let $D$ be an  effective Cartier divisor on $X$,  and $Y$ be a closed subscheme  of $X$.  Define $n_Y(D)$ to be the largest nonnegative integer such that the scheme-theoretic inclusion $n_Y(D) Y\subset D$ holds.
\end{definition}

\begin{lemma}[cf. {\cite[Lemma 2.17]{HLX}}]\label{orderlemma}
Let $D,D'$ be effective Cartier divisors on a projective variety $X$ of dimension $\geq 2$.
Let $P$ be a (reduced) point in $D$  and let $Y=\alpha P$ for some positive integer $\alpha$. Suppose there exists a regular sequence $f_1,\ldots, f_m$ in the local ring  $\mathcal{O}_{X,P}$  such that $D$ is locally defined by  $\mathbf{f}^A:=f_1^{n_1}\cdots f_m^{n_m}$, where $n_i$ are nonnegative integers and $P$ is defined by the ideal $(f_1,\ldots,f_m)$. Then the following hold:
\begin{enumerate}
\item[{\rm(i)}] $n_P(D + D')=n_P(D) + n_P(D')$ and  $n_P(D) = |A|.$
\item[{\rm(ii)}]  $-2<n_Y(D + D')- n_Y(D) - n_Y(D')<2 $, and $ n_Y(D) =\left\lfloor \frac{|A|}{\alpha} \right\rfloor.$
\end{enumerate}
\end{lemma}
   
\begin{proof}
 Let ${\mathscr I}_P\subset \mathcal{O}_{ X, P}$ be the  ideal sheaf of $P$. Let $n=n_P(D+D')$. Suppose that $D$ and $D'$ are represented by $\mathbf{f}^A$ and ${\tilde f}$ locally at $P$. Then since $nP \subset D+D'$, we have $\mathbf{f}^A{\cdot\tilde f}\in {\mathscr I}_P^n$. By Lemma \ref{ideallemma}, ${\tilde f} \in {\mathscr I}_P^{n-| A|}$. On the other hand, since $\mathbf{f}^A \in {\mathscr I}_P^{|A|}$, if ${\tilde f}\in {\mathscr I}_P^{n-|A|+1}$ then  $\mathbf{f}^A{\cdot\tilde f}\in {\mathscr I}_P^{n+1}$, contradicting the definition of $n$. It follows that 
\begin{align*}
n_P(D')=n-|A|. 
\end{align*}
The equality now follows from noting that $n_P(D)=|A|$.
The second assertion can be derived from (i) and that 
$$
 \lfloor{\frac{n_P(D)}{\alpha}}\rfloor\le n_Y(D)\le  \frac{n_P(D)}{\alpha}.
$$
\end{proof}
 
\subsection{Proof of Theorem \ref{trungeneral}}
We begin with some preparation.
Let $D_1,\dots,D_q$ be divisors  intersecting properly on $X$, $\mathscr L$ be a big line sheaf on $X$, and
$n=\dim X$.   We may assume that $\beta_i\in \mathbb Q$ are chosen slightly smaller than $\beta( {\mathscr L}, D_{i})$ for $1\le i\le q$.
Choose positive integers $N$ and $b$ such that
\begin{equation}\label{aut_choices}
  \left( 1 + \frac nb \right) \max_{1\le i\le q}
      \frac{\beta_i Nh^0(X, \mathscr L^N) }
        {\sum_{m\ge1} h^0(X, \mathscr L^N(-mD_i))}
    < 1 + \epsilon\;.
\end{equation}

Let
$$\Sigma
  = \biggl\{\sigma\subseteq \{1,\dots,q\}
    \bigm| \bigcap_{j\in \sigma}  D_j\ne\emptyset\biggr\}.$$
For $\sigma\in \Sigma$, let
$$\bigtriangleup_{\sigma}
  = \left\{\mathbf a = (a_i)\in \prod_{i\in\sigma}\beta_i^{-1}\mathbb N
    \Bigm| \sum_{i\in\sigma} \beta_ia_i = b \right\}.$$
For $\mathbf a\in\bigtriangleup_{\sigma}$ as above, we construct a filtration of $H^0(X, \mathscr L^N)$ as follows: 
For  $x\in  \mathbb R^+$, 
one defines  
the ideal ${\mathscr I}_{\mathbf a}(x)$ of ${\mathscr O}_X$ by
\begin{equation}\label{def_I_of_x}
  {\mathscr I}_{\mathbf a}(x)
  = \sum_{{\bf b}} {\mathscr O}_X\left(-\sum_{i\in \sigma} b_iD_{i}\right)
\end{equation}
where the sum is taken for all ${\bf b}\in {\mathbb N}^{\#\sigma}$
with $\sum_{i\in \sigma} a_ib_i\ge bx$. Let
\begin{equation}\label{filtrationof_x}{\mathcal F}(\sigma; {\bf a})_x
  = H^0(X, \mathscr L^N \otimes {{\mathscr I}_{\mathbf a}}(x)),
  \end{equation}
which we regard as a subspace of $H^0(X,\mathscr L^N)$.
We note that there are only finitely many ordered pairs
$(\sigma,\mathbf a)$ with $\sigma\in\Sigma$
and $\mathbf a\in\bigtriangleup_\sigma$.  Let $\mathcal B_{\sigma; {\bf a}}$ be a basis of
$H^0(X, \mathscr L^N)$ adapted to the above filtration $\{{\mathcal F}(\sigma; {\bf a})_x\}_{x\in\mathbb R^+}$.

  For a basis ${\mathcal B}$ of $H^0(X, {\mathscr L}^N)$, denote by $(\mathcal B)$  the sum of the divisors $(s)$
for all $s\in\mathcal B$.  
We now state the following main lemma in \cite{RV20}.

\begin{lemma}\label{aut_main_lemma} {\cite[Lemma 6.8]{RV20}}
With the above notations, we have
\begin{equation}\label{aut_main_step}
  \bigvee_{\substack{\sigma\in\Sigma \\ \mathbf a\in\Delta_\sigma}}
      (\mathcal B_{\sigma;\mathbf a})
    \ge \frac b{b+n}\left(\min_{1\le i\le q}
      \frac{\sum_{m=1}^\infty h^0(X,\mathscr L^N(-mD_i))}{\beta_i}\right)
      \sum_{i=1}^q \beta_iD_i\;.
\end{equation}
\end{lemma}
Here, the notation ``$\bigvee$" is referred to the least upper bound with respect the partial order on the (Cartier) divisor group of $X$  by the relation $D_1\le D_2$ if $D_2-D_1$ is effective.


\begin{proof}[Proof of Theorem \ref{trungeneral}]

For each  $z=re^{i\theta}\in \mathbb C$, let $I_z:=\{z(1),\hdots,z(n-1)\}$ be an index subset of $\{1,\hdots,q\}$ such that 
$$
\sum_{i=1}^{n-1}\lambda_{D_{z(i)}}(f(z))\ge  \sum_{j\in I} \lambda_{D_{j}}(f(z)) 
$$
for any index subset $I$ of $\{1,\hdots,q\}$ with cardinality $n-1$.

Let $Y_z\subset \bigcap_{i\in I_z} D_i $ be a regular sequence of closed subschemes such that $\beta(\mathcal L, Y_z)\lambda_{Y_z}(f(z))\ge \beta(\mathcal L, Y)\lambda_{Y}(f(z))$ for any closed subscheme $Y \subset \bigcap_{i\in I_z} D_i$.
Since  $D_i$, for $1\le i\le q$, intersect properly,  it follows that $\dim Y_{z}=0$.  We may write $\dim Y_{z}=Y_{1,z}+\cdots+Y_{k,z}$, where $Y_{1,z},\hdots,Y_{k,z}$ are closed subschemes supported at distinct points in $X$.  Noting that  $\beta(\mathcal L, Y_{j,z})\ge \beta(\mathcal L, Y_{z})$ for each $j$ and that $\lambda_{Y_{z}}(P)=\max_j \lambda_{Y_{j,z}}(P)+O(1)$ for all $P\in X\setminus Y_{z}$, we may assume that $Y_{z}$ is supported at a single point.

Let $N$ be a sufficiently large positive integer, which will be chosen and fixed later. Let ${\mathscr I}_z$ denote the ideal sheaf associated with $Y_z$.
We consider the filtration $\mathcal{F}(Y_z)$,  given by order of vanishing along $Y_z$ as follows:
$$
H^0(X, {\mathscr L}^N)\supset H^0(X, {\mathscr L}^N\otimes {\mathscr I}_z )\supset H^0(X, {\mathscr L}^N\otimes {\mathscr I}_z^2 )\supset \cdots.
$$ 
In addition to the criteria for $N$ and $b$ stated at the beginning to satisfy \eqref{aut_choices}, we may assume that $N$ is sufficiently large so that
\begin{align}\label{filtrationbeta}
 \sum_{i=1}^{\infty} h^0(X, {\mathscr L}^N\otimes {\mathscr I}_z^i )
 \ge \frac{N\,h^0(\mathscr L^N)}{1+\epsilon}\, \beta({\mathscr L}, Y_z),
\end{align}
where $h^0(\mathscr L^N)=h^0(X,\mathscr L^N)$.

We will now take another filtration.
Let
$$\Sigma_z
  = \biggl\{\sigma_z\subseteq I_z  
    \bigm| \bigcap_{j\in \sigma_z}  D_{z(j)}\ne\emptyset\biggr\},
$$
and
$$\Delta_{\sigma_z}=\bigg\{\mathbf{a}_z=(a_i) \in \prod_{i \in \sigma_z}\beta_i^{-1}\mathbb{N}\mid \sum_{i \in \sigma_z}\beta_i a_i=b\bigg\}.$$
For any $\mathbf{a}_z \in \Delta_{\sigma_z}$ and $x \in \mathbb{R}^{+}$, define
$${\mathscr I}_{\mathbf{a}_z}(x)
  = \sum_{{\bf b}} {\mathscr O}_X\left(-\sum_{i\in \sigma_z} b_iD_{z(i)}\right),$$
where the sum is taken for all $\mathbf{b} \in \mathbb{N}^{\#\sigma_z}$ with $\sum_{i \in \sigma_z}a_ib_i \geq bx$. Then the second filtration of $H^0(X,\mathscr{L}^N)$ is defined by
$${\mathcal F}(\sigma_z; {\bf a}_z)_x
  = H^0(X, \mathscr L^N \otimes {{\mathscr I}_{\mathbf{a}_z}}(x)).$$

We note that there are only a finite number of $(\sigma_z, \mathbf{a}_z)$ and $Y_z$ since the number of divisor $D_i$ is finite. For each pair $(\sigma_z, \mathbf{a}_z)$, by Lemma 2.18 in \cite{HLX}, there exists a basis 
$ \mathcal{B}_{\sigma_z; \mathbf{a}_z} 
 =\{s_{1,z},\ldots, s_{h^0(\mathscr{L}^N),z}\}$ of $H^0(X, {\mathscr L}^N)$ 
adapted to both filtrations $\mathcal {F}(\sigma_z; {\bf a}_z)_x$ and $\mathcal{F}(Y_z)$.

 By Lemma \ref{aut_main_lemma}, we have the estimate
\begin{align}\label{filtration1}
 \bigvee_{\substack{\sigma_z\in\Sigma_z \\ \mathbf a_z \in \Delta_{\sigma_z}}}( \mathcal B _{\sigma_z;\mathbf{a}_z})
   & \ge \frac b{b+n}\left(\min_{1\le i\le n-1 }
      \frac{\sum_{m=1}^\infty h^0(X,\mathscr L^N(-mD_{z(i)}))}{\beta_{z(i)}}\right)
      \sum_{i=1}^{n-1} \beta_{z(i)}D_{z(i)}\cr
      & \ge  
      \frac{Nh^0(\mathscr L^N)}{1+\epsilon} 
      \sum_{i=1}^{n-1}\beta_{z(i)}D_{z(i)}.
\end{align}
 For simplicity, we denote 
$\bigvee_{\substack{\sigma_z\in\Sigma_z \\ \mathbf a_z \in \Delta_{\sigma_z}}}( \mathcal B _{\sigma_z;\mathbf{a}_z})$ by $\bigvee(\mathcal{B})_z$. Since the $\beta_{z(i)}$'s are carefully chosen to be  rational numbers slightly smaller than $\beta( {\mathscr L}, D_{z(i) })$,
we may choose $\epsilon>0$ and $N$ such that 
both $N\beta_{z(i)}$ and $\frac{N\beta_{z(i)}}{1+\epsilon}$ are integers for each $i$.


Since each $\mathcal{B}_{\sigma_z;\mathbf{a}_z}$ is adapted to the filtration $\mathcal{F}(Y_z)$, it follows that their least upper bound satisfies
\begin{align}\label{multY}
 \bigvee(\mathcal{B})_z  \supset 
        \sum_{i=1}^{\infty} h^0(X, {\mathscr L}^N\otimes {\mathscr I}_z^i )Y_z
        \end{align}
        as closed schemes.
On the other hand, we can write 
\begin{align}\label{sectiondivisor}
 \bigvee(\mathcal{B})_z =
      \frac{Nh^0(\mathscr L^N)}{1+\epsilon} \big( \sum_{i=1}^{n-1}\beta_{v_i}D_{z(i)}\big)+F_z
\end{align}
for some effective divisor $F_z$. We also note that each term in the above equality is an integral divisor.

By applying Lemma \ref{orderlemma}, we have
\begin{equation*}
    -2<n_{Y_z}\Big( \bigvee(\mathcal{B})_z \Big)-n_{Y_z}\Big(\frac{Nh^0(\mathscr{L}^N)}{1+\epsilon}\big( \sum_{i=1}^{n-1}\beta_{v_i}D_{z(i)}\big)\Big)-n_{Y_z}(F_z)<2.
\end{equation*}
Therefore, by \eqref{multY} and \eqref{filtrationbeta}, we conclude that
\begin{align*}
n_{Y_z}(F_z)&\ge n_{Y_z}\Big( \bigvee(\mathcal{B})_z \Big)- \frac{Nh^0(\mathscr L^N)}{1+\epsilon} \big( \sum_{i=1}^{n-1}\beta_{z(i)}\big)-2\\
&\ge \sum_{i=1}^{\infty} h^0(X, {\mathscr L}^N\otimes {\mathscr I}_z^i )-  \frac{Nh^0(\mathscr L^N)}{1+\epsilon} \big( \sum_{i=1}^{n-1}\beta_{z(i)}\big) -2\\
&\ge \frac{Nh^0(\mathscr L^N)}{1+\epsilon} \beta( {\mathscr L}, Y_v)-  \frac{Nh^0(\mathscr L^N)}{1+\epsilon} \big( \sum_{i=1}^{n-1}\beta_{z(i)}\big) -2.\quad 
\end{align*}
Consequently,
\begin{align}\label{Fv}
\lambda_{F_z}(P)&\ge n_{Y_z}(F_z)\cdot\lambda_{Y_z}(P)+O(1)\cr
&\ge \frac{Nh^0(\mathscr L^N)}{1+\epsilon} \Big(\beta( {\mathscr L}, Y_z)-    \sum_{i=1}^{n-1}\beta_{z(i)}  -\frac{2(1+\epsilon)}{Nh^0(\mathscr{L}^N)}\Big)\lambda_{Y_z}(P)+O(1).
\end{align}
Together with \eqref{sectiondivisor} and \cite[Proposition 4.13]{RV20}, we have 
\begin{align}\label{lambdaineq}
 &\sum_{i=1}^{n-1}\beta_{z(i)}\lambda_{D_{z(i)}}(P) 
 +\Big( \beta( {\mathscr L}, Y_z)-   \sum_{i=1}^{n-1 }\beta_{z(i)}  -\frac{2(1+\epsilon)}{Nh^0(\mathscr{L}^N)}\Big)  \lambda_{Y_z}(P)\cr
&\le    \sum_{i=1}^{n-1}\beta_{z(i)}\lambda_{D_{z(i)}}(P) +\frac{1+\epsilon}{Nh^0(\mathscr L^N)} \cdot\lambda_{F_z}(P))+O(1)\cr
&\le\frac{1+\epsilon}{Nh^0(\mathscr L^N)} \max_{\substack{\sigma_z\in\Sigma_z \\ \mathbf a_z \in \Delta_{\sigma_z}}} 
 \lambda_{ \mathcal B _{\sigma_z;\mathbf{a}_z}}(P)+O(1).
\end{align}

As the set  $\{(\sigma_z, \mathbf{a}_z)\in (\Sigma_z, \Delta_{\sigma_z})\,:\,z\in\mathbb C\}$ is finite, we may write  
$$\{ \mathcal{B}_{\sigma_z; \mathbf{a}_z} \,:\, \sigma_z\in\Sigma_z,\, \mathbf a_z \in \Delta_{\sigma_z}, \,  z\in\mathbb C\}=\{ \mathcal B_1,\cdots, \mathcal B_{T_1}\}.
$$
Let $\bigcup_{i=1}^{T_1}\mathcal B_i=\{s_1, \dots, s_{T_2}\}.$ 
We can apply  Theorem \ref{gsmt2}, for sections $s_j$, $1\le j\le T_2$, on the right hand side of \eqref{lambdaineq}. Then there exists a proper  Zariski-closed subset such that for any non-constant holomorphic map $f:\mathbb{C}\to X$ such that  $f(\mathbb{C})\not\subset Z$, 
 \begin{align*}
&\int_0^{2\pi} \max_{I} \Bigg(
    \sum_{i \in I} \beta(\mathscr{L}, D_i) \lambda_{D_i}(f(re^{i\theta})) 
    + \left( \beta(\mathscr{L}, Y_I) - \sum_{i \in I} \beta(\mathscr{L}, D_i) \right) \lambda_{Y_I}(f(re^{i\theta}))
\Bigg) \, \frac{d\theta}{2\pi} \\
&\leq_{\operatorname{exc}} (1 + \varepsilon) T_{\mathscr{L}, f}(r).
\end{align*}
 This completes the proof.
\end{proof}

\section{Proof of Theorem \ref{mainTheorem2}}\label{hyperbolic}
Before proceeding to the proof of Theorem \ref{mainTheorem2}, we present two examples demonstrating the necessity of the conditions on $\beta$ constant and the local intersection multiplicities. 
The first example shows that the condition \eqref{betacond} on the $\beta$-constant in Theorem \ref{mainTheorem2} is not only essential but also sharp.
\begin{example}
    Let the quadric surface $Q \subset \mathbb{P}^3_{\mathbb{C}}$ be defined by the equation $xy=uw$, where $[x:y:u:w]$ are the homogeneous coordinates on $\mathbb{P}^3$. The surface $Q$ is isomorphic to $\mathbb{P}^1 \times \mathbb{P}^1$ by \cite[Ex. I.2.15]{Har}. Consider a family of   curves  on $Q$,  parametrized by $\lambda \in \mathbb{C}$, given by:
    \begin{align*}
       D_{\lambda}:=\{(\lambda t^2:1:\lambda t:t): t\in\mathbb C\}\subset  Q .
    \end{align*}
    Note that for each $\lambda$, the curve $D_{\lambda}$ is cut out on an affine piece of $Q$ by the hyperplane $w=\lambda u$. Let $\tilde{D}_{\lambda}$ be the projectivization of $D_{\lambda}$, which is  a divisor of type $(1,1)$ on $Q\sim \mathbb{P}^1 \times \mathbb{P}^1$. A direct computation shows that
    \begin{equation*}
        \bigcap_{\lambda \in \mathbb{C}} \tilde{D}_{\lambda}= \{(1:0:0:0),(0:1:0:0)\}
    \end{equation*}
Now fix three distinct values $\lambda_1,\lambda_2,\lambda_3 \in \mathbb{C}$, and for  $\lambda\notin \{\lambda_1,\lambda_2,\lambda_3\}$, define the map
    \begin{align*}
        f_{\lambda} : \mathbb{C}  \to Q \setminus (\tilde{D}_{\lambda_1}\cup \tilde{D}_{\lambda_2}\cup \tilde{D}_{\lambda_{3}}), \quad
        z  \mapsto (\lambda  e^{2z}:1:\lambda e^z:e^z).
    \end{align*}
    The union of the images of these maps  $f_{\lambda}$  is a Zariski-dense set in $Q$.
  As noted in \cite[Example~5.12]{HLX}, for any point $P \in Q$,  the beta constant satisfies  
\begin{equation*}
        \beta(  \tilde{D}_{\lambda},P) =1. 
    \end{equation*}
Moreover, by computing  intersection numbers at the triple intersection points  $(1:0:0:0)$ and $(0:1:0:0)$, we find:
    \begin{equation*}
        ( \tilde{D}_{\lambda_i} .  \tilde{D}_{\lambda_j})_Q =2 > \frac{4}{9}\cdot 4 
    \end{equation*}
    for $1\le i\ne j\le 3$,  thereby violating the intersection condition \eqref{localmultip} in Remark \ref{localcondition}.
\end{example}
      
In the preceding example, we observed that the union of the images of the maps  $f_{\lambda}$  is Zariski-dense in $Q$.  However, each individual map $f_{\lambda}$ is algebraically degenrate for every $\lambda$.  The following example illustrates that the intersection condition \eqref{localmultip} in Remark \ref{localcondition}   is indeed necessary, even for obtaining a weaker conclusion of  Theorem \ref{mainTheorem2}.
\begin{example}
Let $X=\mathbb P^2$, $D_1=[x=0]$, $D_2=[y=0]$, and $D_3=[x^2+y^2+xz=0]$.  Then $D_1\cap D_2=
\{[0:0:1]\}$, $D_1\cap D_3=\{[0:0:1]\}$ and $D_2\cap D_3=\{[0:0:1], [-1:0:1]\}$.  Let $Q=[0:0:1]$.
In this case, we have $2D_1\equiv 2D_2\equiv D_3$ and 
$ (D_1,D_3)_Q=2$, and $(D_1,D_3)=2$.  On the other hand, the  right  hand side of \eqref{localmultip} is  $\frac 49\cdot (D_1,D_3)=\frac{8}9$. Therefore, this example  does not satisfy condition  \eqref{localmultip}.
Now consider the entire curve:
$$f=(1,e^z,e^{z^2}-e^{2z}-1):\mathbb C\to \mathbb P^2\setminus D_1\cup D_2\cup D_3.$$
Clearly, $f$ is algebraically non-degenerate.
\end{example}
 
To begin the proof of Theorem~\ref{mainTheorem2}, we first recall the following result from \cite{HLX}.
\begin{lemma}
\label{exclemma}{\cite[Lemma 3.14]{HLX}}
Let $D$ be an ample effective Cartier divisor on a projective variety $X$ of dimension $n$. Let $P\in X(K)$, and let $\pi:\tilde{X}\to X$ be the blowup at $P$, with associated exceptional divisor $E$.  Then for all sufficiently small positive $\delta\in\mathbb{Q}$,
\begin{align}
\label{exceqn}
\beta(\pi^*D-\delta E, \pi^*D-E)>\frac{1}{n+1}.
\end{align}
\end{lemma}

We will first prove a GCD-type theorem, which will serve as a key ingredient in establishing the hyperbolicity assertion in Theorem~\ref{mainTheorem2}.

\begin{theorem}\label{complexgcd2}
Let $X$ be a complex projective variety  of dimension $n$. 
Let $D_1, \ldots,D_{n+1}$ be effective Cartier  divisors on $X$ such that  there exist positive integers $a_1,\ldots, a_{n+1}$ for which $a_1D_1, \ldots, a_{n+1}D_{n+1}$ are all numerically equivalent to an ample divisor $D$.  
Suppose that  any collection of $n$ divisors among  $D_1, \ldots,D_{n+1}$ intersect properly, and that 
\begin{align*}
\bigcap_{i=1}^{n+1}D_i\neq \emptyset.
\end{align*}
Define
\begin{align*}
\beta_0=\min_{\substack{|I|=n\\Q\in \bigcap_{i=1}^{n+1}D_i }} \beta(D,(\bigcap_{i \in I} a_iD_i)_Q),
\end{align*}
where $I$ range over all possible subsets of $\{1,\ldots,n+1\}$ with $|I|=n$.
Furthermore, suppose that for every point $Q\in \bigcap_{i=1}^{n+1}D_i$ and any two subsets $I,I'$ of the indices $\{1,\ldots,n+1\}$ with $|I|=|I'|=n$,  the following inequality holds:
\begin{align}
\label{beta1cond}
(\beta(D,(\bigcap_{i \in I}a_iD_i)_Q)-1)+(\beta(D,(\bigcap_{i \in I'}a_iD_i)_Q)-1)((n+1)\beta_0-n)>0. 
\end{align}
In particular,  note that 
 $\beta_0\geq \frac{n}{n+1}$ and condition \eqref{beta1cond} is satisfied if the stronger assumption
\begin{align}
\label{beta2cond}
\beta(D,(\bigcap_{i \in I}a_iD_i)_Q)>1
\end{align}
holds for all $Q\in  \bigcap_{i=1}^{n+1}D_i  $ and all $I \subset \{1,\ldots,n+1\}, |I|=n$.  
For each $\varepsilon>0$,   there exists a proper Zariski-closed subset $Z\subset X$ such that for any non-constant holomorphic map $f:\mathbb{C}\to X\setminus{\bigcup_{i=1}^{n+1} } D_i$ with $f(\mathbb{C})\not\subset Z$,  we have 
\begin{align*}
T_{f, \,\bigcap_{i=1}^{n+1}a_iD_i}(r) 
        \le \varepsilon  T_{f,D}(r)+O(1). 
\end{align*}
\end{theorem}
 \begin{remark}
The condition \eqref{beta2cond} can be replaced by  the following inequality:
\begin{align}
\label{49ineq2}
(D_i)^I_Q<\left(\frac{n}{n+1}\right)^n(D_i)^I.
\end{align}
This follows from \cite[Lemma~3.11]{HLX}, which implies that
    \begin{align*}
        \beta(D,(\bigcap_{i \in I}a_iD_i)_Q) 
         \geq \frac{n}{n+1} \sqrt[n]{\frac{(a_iD_i)^I}{(a_iD_i)^I_Q}}.
    \end{align*}
 \end{remark}


We now state a refined version of Theorem~\ref{mainTheorem2}.
\begin{theorem}\label{quasi-hyperbolic}
    Under the same assumptions as in the preceding theorem, there exists a proper Zariski-closed subset $Z\subset X$ such that  for any non-constant holomorphic map $f:\mathbb{C}\to X\setminus{\bigcup_{i=1}^{n+1} } D_i$, the image of $f$ is contained in $Z$.
    \end{theorem}
 
 \begin{proof}[Proof of Theorem \ref{complexgcd2}]
The strategy of this proof is to take a convex combination of two applications of Theorem \ref{trungeneral}, taking different sets of $n$ divisors to cut  out schemes with a common point in their support.

For simplicity of notation, we will again write $D_i=a_iD_i$ throughout the proof.    Under this convention, $D_i$'s are numerically equivalent to $D$. Given a point  $Q\in \bigcap_{i=1}^{n+1}D_i $, $z\in\mathbb C$ and  a subset $I \subset \{1,\ldots,n+1\}$ with $|I|=n$,  we introduce the following shorthand notation, which will be used throughout the proof
:\begin{align}\label{shortnotation}
\beta_{I,Q}=\beta(D,(\bigcap_{i \in I} D_i)_{Q}),\quad\text{and}\quad
\lambda_{I,Q}(f(z))&=\lambda_{(\bigcap_{i \in I} D_i)_{Q}}(f(z)). 
\end{align}

Since $D_i\equiv D$, we have $\beta(D,D_i)=\frac{1}{n+1}$ for all $i$, and it follows that for all $Q\in X$ and any $I \subset \{1,\ldots,n+1\}$ with $|I|=n$, we have
\begin{align*}
\beta_{I,Q}\geq \beta(D,\bigcap_{i \in I}D_{i})\geq \sum_{i \in I}\beta(D,D_i)\geq \frac{n}{n+1}.
\end{align*}
In particular, $\beta_0\geq \frac{n}{n+1}$.
Now set 
$$
b = \frac{1}{(n+1)\beta_0-n+1}.
$$ Then $0<b\leq 1$.   By the definition of $\beta_0$, for any $I \subset \{1,\ldots,n+1\}$ with $|I|=n$, and $Q\in  \bigcap_{i=1}^{n+1} D_{i}  $, 
\begin{align}\label{bbeta}
b ( (n+1)\beta_{I,Q}  -n+1 ) \geq 1. 
\end{align}
Define
\begin{align*}
\gamma=(n+1)\min_{\substack{I\neq J,|I|=|J|=n\\Q\in  \bigcap_{i=1}^{n+1} D_{i}  }} b(\beta_{I,Q}-1)+(1-b)(\beta_{J,Q} -1).
\end{align*}
By hypothesis \eqref{beta1cond} and the fact that $b>0$, it follows that $\gamma>0$.  Fix a point $Q\in \bigcap_{i=1}^{n+1}D_i $, and $z\in\mathbb C$.  For any $I \subset \{1,\ldots,n+1\}$ with $|I|=n$, define
\begin{align*}
\beta_{I} :=\beta_{I,Q}, \quad\text{and}\quad
\lambda_{I} :=\lambda_{I,Q}(f(z)). 
\end{align*}
Assume there exist two such index sets  $I,J$, and suppose that $\lambda_{I}\geq \lambda_{J}$.  We then compute
\begin{align}\label{gammabeta}
&b((n+1)\beta_{I}-n+1)\lambda_{I}+(1-b)((n+1)\beta_{J}-n+1)\lambda_{J}\cr
&\geq \lambda_{I}+\left(b((n+1)\beta_{I}-n+1)-1+(1-b)((n+1)\beta_{J}-n+1) \right)\lambda_{J}\cr
&= \lambda_{I}+\left(1+  (n+1)\big(b(\beta_{I}-1)+(1-b)(\beta_{J}-1)  \big) \right) \lambda_{J}\cr
&\geq \lambda_{I}+(1+\gamma)\lambda_{J}.
\end{align}

Let $P=f(z)\in X$.  Choose a permutation
\[
\{1_z,\ldots,(n+1)_z\}=\{1,\ldots,n+1\}
\]
(depending on $z$) such that
\[
\lambda_{D_{1_z}}(f(z)) \ge \lambda_{D_{2_z}}(f(z)) \ge \cdots \ge \lambda_{D_{(n+1)_z}}(f(z)).
\]
Define the two index sets
\[
I_z:=\{1_z,\ldots,n_z\},\qquad 
J_z:=\{1_z,\ldots,(n-1)_z,(n+1)_z\}.
\]
For each $z\in \mathbb C$, there exists a point $Q_z\in \text{Supp} (\bigcap_{i \in I_z} D_{i})$ (depending on $f(z)$) such that
\begin{align*}
\lambda_{\bigcap_{i \in I_z} D_{i}}(f(z))=\lambda_{ (\bigcap_{i \in I_z} D_{i})_{Q_z}}(f(z))+O(1),
\end{align*}
where the constant in the $O(1)$ is independent of $f(z)$.
We claim 
\begin{align}\label{localWeilQz0}
 &\sum_{i=1}^{n+1}\lambda_{ D_{i}}(f(z))+ \gamma \lambda_{\bigcap_{i=1}^{n+1}D_{i}} (f(z)) \cr
& = \sum_{i=1}^{n-1}\lambda_{D_{i_z}}(f(z))  +\lambda_{   I_z,Q_z }(f(z))+  (1+\gamma) \lambda_{ J_z,Q_z }(f(z))+O(1). 
\end{align}

If $Q_z\not\in \text{Supp} \bigcap_{i=1}^{n+1}D_{i}$, then
\begin{align*}
\lambda_{ D_{(n+1)_z}   }(f(z))&=\min \{\lambda_{D_{(n+1)_z} }(f(z)),\lambda_{\bigcap_{i  \in I_z} D_i}(f(z))\}+O(1)\\
&=\min \{\lambda_{D_{(n+1)_z} }(f(z)),\lambda_{   I_z,Q_z }(f(z))\}+O(1)\\
&=\lambda_{(\bigcap_{i \in I_z} D_{i})_{Q_z}\cap D_{(n+1)_z}}(f(z))=O(1),  
\end{align*}
since $({\bigcap_{i \in I_z} D_{i}})_{Q_z}\cap D_{(n+1)_z}$ is empty.  
Then
\begin{align*}
  &\sum_{i=1}^{n+1}\lambda_{ D_{i}}(f(z)) + \gamma \lambda_{\bigcap_{i=1}^{n+1}D_{i}} (f(z)) 
 = \sum_{i=1}^{n}\lambda_{ D_{i_z}}(f(z)) +O(1)\cr
&= \sum_{i=1}^{n-1}\lambda_{D_{i_z}}(f(z))  +\lambda_{   I_z,Q_z }(f(z))+O(1).
\end{align*}
As $\lambda_{J_z,Q_z }(f(z))=O(1)$ in this case, the claim \eqref{localWeilQz0} follows.
When $Q_z\in \text{Supp} (\bigcap_{i=1}^{n+1} D_{i})$, 
we use the estimate
\begin{align*}
\lambda_{ D_{(n+1)_z} } (f(z)) &=\lambda_{(\bigcap_{i=1}^{n+1} D_{i})_{Q_z}} (f(z)) +O(1)\\
&\leq \lambda_{ J_z,Q_z }(f(z)) (f(z) )+O(1) 
\end{align*}
to conclude the calim  \eqref{localWeilQz0}.

Since $\lambda_{   I_z,Q_z }(f(z))\ge \lambda_{ J_z,Q_z }(f(z))+O(1)$,
we may apply \eqref{gammabeta} to \eqref{localWeilQz0} to obtain
\begin{align}\label{localWeilQz3}
 & \sum_{i=1}^{n+1}\lambda_{ D_{i}}(f(z))+ \gamma \lambda_{\bigcap_{i=1}^{n+1}D_{i}} (f(z))  \notag\\
 &\le \bigg(\sum_{i=1}^{n-1}\lambda_{D_{i_z}}(f(z)) \bigg)+b\left((n+1)\beta_{I_z,Q_z}-n+1)\lambda_{I_z,Q_z}(f(z)\right) \notag\\
&\quad+(1-b)\left((n+1)\beta_{J_z,Q_z}-n+1\right)\lambda_{J_z,Q_z}(f(z))+O(1). 
\end{align}
Let $\varepsilon>0$.  We will apply Theorem \ref{trungeneral} twice.
First, for each point $z\in\mathbb C$ we take $I=\{1_z,\ldots,(n-1)_z \}$ and $Y_z=(\bigcap_{i \in I_z} D_{i})_{Q_z} $.  Apply the theorem with this choice and multiply the resulting inequality by $b$.
Second, again for each point $z\in\mathbb C$ we take  $I=\{1_z,\ldots,(n-1)_z \}$ but now let $Y_z=(\bigcap_{i \in J_z} D_{i})_{Q_z} $   and multiply the resulting inequality by $1-b$.
Combining these two applications with \eqref{localWeilQz3}, we obtain that there exists a proper Zariski-closed subset  $Z\subset X$, independent of $f$, such that  
\begin{align}\label{proxi2}
  \sum_{i=1}^{n+1}m_f(D_i,r)  +\gamma\cdot m_f\bigg( {\bigcap_{i=1}^{n+1} }  D_{i} ,r\bigg)   
    \leq_{\operatorname{exc}} (n+1+\varepsilon)T_{D,f}(r) 
\end{align}
if $f(\mathbb{C})\not\subset Z$. 
In addition, if the image of $f$ does not intersects any $D_i$, $1\le i\le n+1$, then  $N_f(D_i,r) =O(1)$ for  $1\le i\le n+1$, and hence $N_f({\bigcap_{i=1}^{n+1} }  D_{i},r) =O(1)$.
 Note that 
\begin{equation*}
    N_f(D_i,r)+ m_f(D_i,r) =T_{D_i,f}(r)  
\ge \left(1-\frac{\epsilon}{n+1}\right) T_{D,f}(r)+O(1),
\end{equation*}
since the divisors $D_i$ are numerically equivalent to $D$. Therefore, \eqref{proxi2} yields
\begin{align}\label{ht}
(n+1-\varepsilon)T_{D,f}(r)   
+\gamma\, T_{\bigcap_{i=1}^{n+1}D_{i},f}(r)   
\leq_{\operatorname{exc}} (n+1+\varepsilon)T_{D,f}(r).
\end{align}
 Consequently,
\begin{align*} 
 T_{{\bigcap_{i=1}^{n+1} }  D_{i},f}(r)   
    \leq_{\operatorname{exc}} \frac{2\varepsilon}{\gamma} T_{D,f}(r).
\end{align*}
\end{proof}

\begin{proof}[Proof of Theorem \ref{quasi-hyperbolic}]
Let $Q$ be a point in $\text{Supp} (\bigcap_{i=1}^{n+1}D_i)$. Let $\pi:\tilde{X}\to X$ be the blowup at $Q$, with exceptional divisor $E$.  Let $f:\mathbb{C}\to X\setminus{\bigcup_{i=1}^{n+1} } D_i$ be a nonconstant analytic map.  Then $f$ lifts to an analytic map 
$$
\tilde{f}:\mathbb C\to \tilde{X}\setminus{\bigcup_{i=1}^{n+1} } \pi^*D_i
$$ 
such that $\pi\circ \tilde f=f$.  

 Define the effective Cartier divisors 
\begin{align}\label{stricttransform}
\tilde{D_i}=\pi^*D_i-E, \quad i=1,\ldots,n+1.
\end{align}

Let $P=\tilde f(z)\in \tilde X$ for some $z\in \mathbb C$.
Let $ \{1_z,\ldots,(n+1)_z \}= \{1,\ldots,n+1\}$   be such that
\begin{align*}
  \lambda_{ \tilde D_{1_z}}(\tilde f(z))  \geq \cdots \geq\lambda_{\tilde  D_{(n+1)_z}}(\tilde f(z))  .
\end{align*}
Then 
\begin{align}\label{localWeilQz4}
 \sum_{i=1}^{n+1}\lambda_{ \tilde D_{i}}(\tilde f(z))= \bigg(\sum_{i=1}^{n-1}\lambda_{\tilde D_{i_z}}(\tilde f(z)) \bigg) +\lambda_{\bigcap_{i =1}^{n} \tilde D_{i_z} }(\tilde f(z))+\lambda_{ \bigcap_{i =1}^{n+1} \tilde D_{i}  } (\tilde f(z) )+O(1). 
\end{align}
By \eqref{stricttransform} and the functorial properties of local Weil-functions, we have 
\begin{align}\label{localWeilQz5}
 \sum_{i=1}^{n+1}\lambda_{ \tilde D_{i}}(\tilde f(z))&= \sum_{i=1}^{n+1} \lambda_{D_{i}}(f(z))-(n+1) \lambda_{Q}(f(z))+O(1)\cr
 &\ge  \sum_{i=1}^{n+1} \lambda_{D_{i}}(f(z))-(n+1) \lambda_{{\bigcap_{i=1}^{n+1} } D_i}(f(z))+O(1). 
\end{align}
Let $\varepsilon>0$ be a sufficiently small number to be determined later.
By applying Theorem \ref{complexgcd2}, there exists a proper Zariski-closed subset $Z\subset X$ such that  
\begin{align}\label{leftproxi}
  \sum_{i=1}^{n+1}m_{\tilde f}( \tilde D_{i},r)&\ge  (n+1-(n+2)\varepsilon )T_{D,f}(r).
 \end{align}
if the image of $f$ is not contained in $Z$.   We note that $N_{ f}(  D_{i},r)=O(1)$.
For the same reason, we have 
\begin{align}\label{rightproxi1}
 m_{\tilde f}\bigg( \bigcap_{i =1}^{n+1} \tilde D_{i} ,r\bigg)\le  m_{\tilde f}\bigg( \bigcap_{i =1}^n \pi^*D_{i}, r\bigg)
= m_{  f}\bigg( \bigcap_{i =1}^{n+1} D_{i} ,r\bigg)+O(1)\le  \varepsilon T_{D,f}(r)
\end{align}
if the image of $f$ is not contained in $Z$.

We now estimate the first two terms on the right hand side of \eqref{localWeilQz4}.
\begin{align*}
\bigcap_{i=1}^n\tilde D_{i_z}=Y_{0,z}+Y_{1,z},
\end{align*}
where $\Supp (\pi(Y_{1,z}))=Q$, $\dim Y_{0,z}=0$, and $Y_{0,z}\cap E=\emptyset$. By adjusting $\varepsilon$ according to the multiplicity of $Y_{1,z}$, we have (similar  to deriving \eqref{leftproxi})
\begin{align}\label{proxiY1}
m_{\tilde f}({Y_{1,z}},r)\le \varepsilon T_{D,f}(r)
\end{align}
if the image of $f$ is not contained in $Z$.  

By Lemma \ref{exclemma}, let $\delta\in\mathbb{Q}, \delta>0$, be so that $\beta(\pi^*D-\delta E,\pi^*D-E)>\frac{1}{n+1}$ holds, and let 
\begin{align*}
\gamma'=\beta(\pi^*D-\delta E, \pi^*D-E)-\frac{1}{n+1}>0.
\end{align*}
Note that $\beta(\pi^*D-\delta E, \pi^*D-E)=\beta(\pi^*D-\delta E, \tilde D_{i})$ for $1\le i\le n+1$, and
\begin{align*}
\beta(\pi^*D-\delta E, Y_{0,z})\geq \sum_{i=1}^{n}\beta(\pi^*D-\delta E, \tilde D_{i_z}). 
\end{align*}
Consequently,
\begin{align}\label{localWeilQz5} 
 \left(\frac{1}{n+1}+\gamma'\right) &\bigg(\sum_{i=1}^{n-1}\lambda_{\tilde D_{i_z}}(\tilde f(z)) +\lambda_{Y_{0,z}}(\tilde f(z))\bigg)   
\le \bigg(\sum_{i=1}^{n-1}\beta(\pi^*D-\delta E, \tilde D_{i_z})\lambda_{\tilde D_{i_z}}(\tilde f(z)) \bigg)\notag\\
&+\left(\beta(\pi^*D-\delta E, Y_{0,z})-\sum_{i=1}^{n-1}\beta(\pi^*D-\delta E, \tilde D_{i_z})\right)\lambda_{Y_{0,z}}(\tilde f(z)). 
\end{align}

We now apply Theorem \ref{trungeneral} for $\{1_z,\hdots,(n-1)_z\}$ and $Y_{0,z}$,  and use
\eqref{rightproxi1} and \eqref{proxiY1} for \eqref{localWeilQz4}.  Then there exists a proper Zariski-closed subset  $W \subset \tilde X$, independent of $f$, such that  
\begin{align}\label{proxiF}
  \sum_{i=1}^{n+1}m_{\tilde f}( D_i,r)  
    &\leq_{\operatorname{exc}} \frac{(n+1)(1+\varepsilon)}{ 1+\gamma'}T_{\pi^*D-\delta E,\tilde f}(r)+2\varepsilon T_{D,f}(r) \cr
    & \leq_{\operatorname{exc}} \bigg(\frac{(n+1)(1+\varepsilon)}{ 1+\gamma'}+2\varepsilon\bigg) T_{D,f}(r).\end{align}
 if $ f(\mathbb{C})\not\subset \pi(W).$   By taking $\varepsilon<\min\bigg\{\frac12, \frac{\gamma'}{5(n+1)}\bigg\}$, we can derive from \eqref{leftproxi} and \eqref{proxiF} that 
$c_1T_{D,f}(r)\leq_{\operatorname{exc}}  O(1)$ for some positive real $c_1$ if the image of $f$ is not contained in $Z\cup\pi(W)$.  Let $A$ be an ample divisor.  Then there exists a positive real $c_2$ and a proper Zariski-closed subset $Z_2\subset X$ such that $T_{D,f}(r)\ge c_2 T_{A,f}(r)$ if the image of $f$ is not contained in $Z_2$.  Therefore, we may conclude from the inequality $T_{A,f}(r) \leq_{\operatorname{exc}}  O(1)$ that the image of $f$ must contained in  $Z\cup Z_2\cup \pi(W)$.
\end{proof}

\section{Proof of Theorem \ref{MainThmgcd}}\label{GCD}
We  now state a refined version of Theorem \ref{MainThmgcd}.

\begin{theorem}\label{complexgcd}
Let $D_1, \ldots,D_{n+1}$ be effective  divisors intersecting properly on a complex projective variety $X$ of dimension $n$. Suppose that there exist positive integers $a_1,\ldots, a_{n+1}$ such that $a_1D_1, \ldots, a_{n+1}D_{n+1}$ are all numerically equivalent to an ample divisor $D$.  
Suppose that for some index set $I_0 \subset \{1,\ldots,n+1\}$ with $|I|=n$ such that for all $Q\in X$, 
\begin{align}\label{beta23cond}
\beta(D,(\bigcap_{i \in I_0}a_iD_i)_Q)  > \frac{n}{n+1}. 
\end{align}
In particular, \eqref{beta23cond} is satisfied if $\bigcap_{i\in I_0}D_i$ contains more than one point. 
For each $\varepsilon>0$,   there exists a proper  Zariski-closed subset $Z\subset X$  such that for any non-constant holomorphic map $f:\mathbb{C}\to X\setminus{\bigcup_{i=1}^{n+1} } D_i$ such that  $f(\mathbb{C})\not\subset Z$,  we have 
\begin{align*}
T_{\bigcap_{i \in I_0}a_iD_i, f}(r) 
        \le \varepsilon  T_{D,f}(r). 
\end{align*}
\end{theorem}

 \begin{proof}[Proof of Theorem \ref{complexgcd}]
 For simplicity of notation, we will use $D_i=a_iD_i$ throughout the proof.    Under this convention, the divisors \(D_1,\ldots,D_{n+1}\) are numerically equivalent to \(D\), and
\[
\beta(D,D_i)=\frac{1}{n+1}
\qquad \text{for all } i.
\]
Moreover, for every \(Q\in X\) and every subset \(I \subset \{1,\ldots,n+1\}\) with \(|I|=n\), we have
\begin{align}
\label{23ineq}
\beta\bigl(D,(\textstyle\bigcap_{i \in I} D_i)_Q\bigr)
\geq \beta\bigl(D,\textstyle\bigcap_{i \in I} D_i\bigr)
\geq \sum_{i \in I}\beta(D, D_i)
\geq \frac{n}{n+1}.
\end{align}
  
  Let  $I_0$ be a  subset of  $\{1,\ldots,n+1\}$ with $|I|=n$ such that \eqref{beta23cond} is satisfied for all $Q\in X$.
Then
\begin{align}\label{gamma0}
\gamma =\min_{Q\in X }\{(n+1)\beta(D,(\bigcap_{i \in I_0} D_i)_Q) -n\}>0.
\end{align}
 We further observe that  if $\bigcap_{i\in I_0}D_i$ contains more than one point, then the first inequality in \eqref{23ineq} is strict:
\begin{align*}
\beta(D,(\bigcap_{i \in  I_0} D_i)_Q)> \beta(D,\bigcap_{i \in I} D_i)\geq \sum_{i \in I}\beta(D, D_i) \geq \frac{n}{n+1}.
\end{align*}
Hence, the assumption \eqref{beta23cond} holds as claimed.

For any point $P\in X $,  there is a point $Q\in   \bigcap_{i \in I} D_i$ (depending on $P$  and $I$) such that
\begin{align}\label{pointeq}
\lambda_{\bigcap_{i \in I} D_i }(P)=\lambda_{(\bigcap_{i \in I} D_i)_Q}(P)+O(1)
\end{align}
where the constant implied by $O(1)$ is independent of $P$.

Let $f(z)\in X$ for some $z\in \mathbb C$.
Let $I_z=\{1_z,\ldots,n_z\} \subset \{1,\ldots,n+1\}$ and let $(n+1)_z  \in\{1,\ldots,n+1\} \setminus I_z$ be such that
$$
\lambda_{ D_{1_z}}(f(z))\geq \cdots \geq \lambda_{ D_{(n+1)_z} }(f(z)).
$$
Then
\begin{align}\label{localWeil}
  \sum_{i=1}^{n+1}\lambda_{ D_{i_z}}(f(z))  
 &= \sum_{i=1}^{n}\lambda_{ D_{i_z}}(f(z)) +O(1)\cr
&= \sum_{i=1}^{n-1}\lambda_{D_{i_z}}(f(z))  +\lambda_{\bigcap_{i=1}^n  D_{i_z}}(f(z))+O(1).
\end{align}
Noting that 
\begin{align}\label{intertrivial} 
\lambda_{\bigcap_{i \in I}  D_{i }}  (f(z)) =O(1)\quad  \text{if $I \neq I_z$},
\end{align}
where $I$ is an index subset of $\{1,\hdots,n+1\}$ with $n$ elements.
 
 For any $z\in \mathbb C$, we let
 \begin{align*}
 \gamma_z=
 \begin{cases}
 \gamma & \text{if } I_z=I_0,\\
 0 & \text{otherwise}.
 \end{cases}
 \end{align*}
 Furthermore, we choose a point
$Q_z \in   \left( \bigcap_{i \in I_z} D_i \right)$
such that \eqref{pointeq} holds with $I = I_z$.

Then by \eqref{localWeil},  \eqref{intertrivial} and \eqref{pointeq}, we have 
\begin{align}\label{localWeil2}
& \sum_{i=1}^{n+1}\lambda_{ D_{i_z}}(f(z))+\gamma\lambda_{\bigcap_{i\in I_0}   D_{i}}(f(z))\notag\\
 &\le \sum_{i=1}^{n-1}\lambda_{D_{i_z}}(f(z)) + (1+\gamma_z)   \lambda_{\bigcap_{i \in I_z} D_{i} } (f(z))+O(1)\notag\\
 &= \sum_{i=1}^{n-1}\lambda_{D_{i_z}}(f(z))+(1+\gamma_z)  \lambda_{(\bigcap_{i \in I_z} D_{i})_{Q_z}}( f(z))+O(1)\notag\\
 &\le \sum_{i=1}^{n-1}\lambda_{D_{i_z}}(f(z)) + (n+1) \big(  \beta(D, ({\bigcap_{i \in I_z } }D_{i})_{Q_z}) -
 \frac{ n-1}{n+1} \big)\cdot \lambda_{(\bigcap_{i \in I_z} D_{i})_{ Q_z}}( f(z))+O(1)\notag\\
 & = \sum_{i=1}^{n-1} \lambda_{D_{i_z}}(f(z))  
 + (n+1) \big(  \beta(D, ({\bigcap_{i \in I_z } }D_{i})_{Q_z}) -\sum_{i=1}^{n-1}\beta(D, D_{i_z})  \big)\cdot \lambda_{(\bigcap_{i \in I_z} D_{i})_{ Q_z}}(f(z)) +O(1).
\end{align}

By integrating \eqref{localWeil2} and applying  Theorem \ref{trungeneral} with $\varepsilon>0$,  there exists a proper  Zariski-closed subset $Z\subset X$ independent of $f$ such that  
\begin{align}\label{proxi}
  \sum_{i=1}^{n+1}m_f(D_i,r)  +\gamma\cdot m_f\bigg({\bigcap_{i\in I_0}   D_{i}},r\bigg)   
    \leq_{\operatorname{exc}} (n+1+\varepsilon)T_{D,f}(r) 
\end{align}
provided that $f(\mathbb{C})\not\subset Z$.
In addition, assume that the image of $f$ does not intersect  any of the $D_i$, $1\le i\le n+1$.
Then  $N_f(D_i,r) =O(1)$ for  $1\le i\le n+1$, and hence $N_f({\bigcap_{i\in I_0}  } D_{i},r) =O(1)$.
Since $N_f(D_i,r)+ m_f(D_i,r) =T_{D_i,f}(r)=T_{D,f}(r)+O(1)$,
 \eqref{proxi} yields
\begin{align}\label{ht}
(n+1-\varepsilon)T_{D,f}(r)   +\gamma\cdot T_{{\bigcap_{i\in I_0}}   D_{i},f}(r)   
    \leq_{\operatorname{exc}} (n+1+\varepsilon)T_{D,f}(r).
\end{align}
Consequently,
\begin{align*} 
 T_{{\bigcap_{i\in I_0} }  D_{i},f}(r)   
    \leq_{\operatorname{exc}} \frac{2\varepsilon}{\gamma} T_{D,f}(r).
\end{align*}
 \end{proof}

\end{document}